\documentclass[11pt]{amsart}

\usepackage{amssymb}
\usepackage{verbatim}

\begin{document}

\newtheorem{theorem}{Theorem}    
\newtheorem{proposition}[theorem]{Proposition}
\newtheorem{conjecture}[theorem]{Conjecture}
\def\theconjecture{\unskip}
\newtheorem{corollary}[theorem]{Corollary}
\newtheorem{lemma}[theorem]{Lemma}
\newtheorem{observation}[theorem]{Observation}
\theoremstyle{definition}
\newtheorem{definition}{Definition}
\newtheorem{remark}{Remark}
\newtheorem{question}{Question}
\def\thequestion{\unskip}
\newtheorem{example}{Example}
\def\theexample{\unskip}
\newtheorem{problem}{Problem}

\numberwithin{theorem}{section}
\numberwithin{definition}{section}
\numberwithin{equation}{section}
\numberwithin{remark}{section}

\def\intprod{\mathbin{\lr54}}
\def\reals{{\mathbb R}}
\def\integers{{\mathbb Z}}
\def\complex{{\mathbb C}\/}
\def\naturals{{\mathbb N}\/}
\def\rationals{{\mathbb Q}\/}
\def\distance{\operatorname{distance}\,}
\def\degree{\operatorname{degree}\,}
\def\kernel{\operatorname{kernel}}
\def\dim{\operatorname{dimension}\,}
\def\Span{\operatorname{span}\,}
\def\nullspace{\operatorname{nullspace}\,}
\def\ZZ{ {\mathbb Z} }
\def\e{\varepsilon}
\def\p{\partial}
\def\rp{{ ^{-1} }}
\def\Re{\operatorname{Re\,} }
\def\Im{\operatorname{Im\,} }
\def\ov{\overline}
\def\bx{{\bf{x}}}
\def\eps{\varepsilon}
\def\lt{L^2}
\newcommand{\norm}[1]{ \|  #1 \|}

\def\scriptx{{\mathcal X}}
\def\scriptb{{\mathcal B}}
\def\scripta{{\mathcal A}}
\def\scriptk{{\mathcal K}}
\def\scriptd{{\mathcal D}}
\def\scriptp{{\mathcal P}}
\def\scriptl{{\mathcal L}}
\def\scriptv{{\mathcal V}}
\def\scripti{{\mathcal I}}
\def\scripth{{\mathcal H}}
\def\scriptm{{\mathcal M}}
\def\scripte{{\mathcal E}}
\def\scriptt{{\mathcal T}}
\def\scriptb{{\mathcal B}}
\def\frakg{{\mathfrak g}}
\def\frakG{{\mathfrak G}}

\author{Michael Christ}
\address{
        Michael Christ\\
        Department of Mathematics\\
        University of California \\
        Berkeley, CA 94720-3840, USA}
\email{mchrist@math.berkeley.edu}
\thanks{This research was supported by NSF grants DMS-040126 and DMS-0901569.}

\date{October 28, 2008. Revised July 12, 2011.}

\title[Bounds for multilinear sublevel sets]
{Bounds for multilinear sublevel sets \\ via Szemer\'edi's theorem}


\maketitle

\section{Introduction}
Let $\ell_j:\reals^d\to\reals^{d_j}$ be surjective linear transformations,
let $P:\reals^d\to \reals$ be a real-valued polynomial, and 
let $\eta\in C^1_0(\reals^d)$ be a compactly supported,
continuously differentiable cutoff function.
For $\lambda\in\reals$ define the multilinear oscillatory integral forms
\begin{equation} \label{oscintegraldefn}
\scripti_\lambda(f_1,\cdots,f_n)
=\int_{\reals^d} e^{i\lambda P(y)} \prod_{j=1}^n f_j\circ\ell_j(y)
\eta(y)\,dy.
\end{equation}
Under what conditions do there exist $\delta>0$ and $C<\infty$
such that for all $f_j\in L^\infty(\reals^{d_j})$,
\begin{equation} \label{decay}
|\scripti_\lambda(f_1,\cdots,f_n)|
\le C|\lambda|^{-\delta}\prod_j\|f_j\|_{L^\infty}
\text{ for all $\lambda\in\reals$?}
\end{equation}
Under what conditions does there exist a function $\rho$
satisfying $\rho(\lambda)\to 0$ as $|\lambda|\to\infty$ such that
for all functions $f_j\in L^\infty$,
\begin{equation} \label{slowerdecay}
|\scripti_\lambda(f_1,\cdots,f_n)|
\le \rho(\lambda)\prod_j\|f_j\|_{L^\infty}
\text{ for all $\lambda\in\reals$?}
\end{equation}
The question in the form \eqref{decay} was
posed by Li, Tao, Thiele, and this author in \cite{multiosc},
where an affirmative answer was demonstrated under certain dimensional restrictions.
Nonoscillatory inequalities of the form
$\int \prod_j |f_j\circ\ell_j|\lesssim \prod_j\norm{f_j}_{L^{p_j}}$
have been studied in \cite{bcct1},\cite{bcct2}.
More refined questions about the optimal exponent $\delta$ in \eqref{decay},
and about inequalities with $\prod_j \norm{f_j}_{L^\infty}$
replaced by $\prod_j\norm{f_j}_{L^{p_j}}$, are of interest, but are at present premature.

Oscillatory integral inequalities of this type have been extensively studied 
for $n=2$, where one is dealing with bilinear 
forms $\langle T_\lambda(f_1),f_2\rangle$.
The associated linear operators $T_\lambda$ are commonly known 
as oscillatory integrals of the second type, and
a simple necessary and sufficient condition for \eqref{decay}
to hold (with some unspecified exponent) is known \cite{steinfatbook}.
There is an extensive literature dealing with more specific inequalities
involving $L^p$ norms, in which one seeks the optimal exponent
$\delta$ as a function of exponents $p$. 

For $n\ge 3$, however, there arises a class of singular
oscillatory integrals which have no direct analogues in the 
bilinear case.
These singular cases arise when $d<\sum_j d_j$. 
Generic ordered $n$-tuples of points
$(x_1,\cdots,x_n)\in \times_j \reals^{d_j}$
then do not contribute to the integral
$\scripti_\lambda$, which may alternatively be expressed as
\begin{equation}
\scripti_\lambda(f_1,\cdots,f_j) 
=\int_\Sigma
e^{i\lambda P(x)}
\prod_j f_j(x_j)\tilde\eta(x)\,d\sigma(x)
\end{equation}
for a certain linear subspace $\Sigma\subset \times_j \reals^{d_j}$
of positive codimension. Here $\tilde\eta$ is smooth and has compact
support, and $\sigma$ is Lebesgue measure on $\Sigma$.
Bilinear situations which superficially appear to be singular,
are always reducible to nonsingular ones in a certain sense, 
but this is not so for $n\ge 3$, in general. 
See \S\ref{section:bilinearcase} below for discussion of this point. 

To date little is known about the general singular multilinear case.
Some cases not covered in \cite{multiosc} have been treated,
after the essential completion of the present work,
in \cite{christdosilva} and \cite{christreduction},
but the general case is not accessible by the methods of those papers alone. 
There are indications that, as was emphasized for certain
related bilinear problems in \cite{ccw}, 
\eqref{decay},\eqref{slowerdecay} are
linked to combinatorial issues. 

An obvious necessary condition \cite{multiosc} for \eqref{slowerdecay} 
is that $P$ should be {\em nondegenerate}, relative to $\{\ell_j\}$,
in the sense that $P$ cannot be expressed in the form
$P = \sum_{j=1}^n p_j\circ\ell_j$ for any measurable functions $p_j$;
this is equivalent \cite{multiosc} to there being no such representation
in which $p_j$ are polynomials of degrees not exceeding the degree
of $P$. For the bilinear case $n=2$, 
nondegeneracy of $P$ is indeed sufficient for \eqref{decay}.
The main results of \cite{multiosc} asserted that for $n\ge 3$,
nondegeneracy of $P$ implies \eqref{decay}, under certain
rather restrictive
supplementary hypotheses. In particular, this holds when all $d_j=d-1$,
and it holds when all $d_j=1$ provided that $n<2d$. No example is known
to us in which a nondegenerate polynomial has been shown 
not to satisfy \eqref{decay}, let alone \eqref{slowerdecay},
but the vast majority of cases remain open.

In the present paper we do not answer these basic questions in any cases;
rather, we study a class of weaker inequalities 
\eqref{ineq:sublevel} which
would be implied by \eqref{slowerdecay}. 
We establish these inequalities
for all nondegenerate polynomials satisfying a natural
rationality hypothesis, whereas only quite restricted classes
of polynomials were treated in \cite{multiosc}. 
A second main result sheds additional light
on the meaning of nondegeneracy, 
by establishing its equivalence, under the rationality hypothesis,  
with a formally stronger property, 
which we call finitely witnessed nondegeneracy.
This furnishes an essential link with additive combinatorics.
With this equivalence in hand, the remainder of the proof
is a nearly direct application of
a generalization of Szemer\'edi's theorem due to 
Furstenberg and Katznelson \cite{furstkatz}.

This second main result is intended to serve as an essential step 
in an attack on oscillatory integral bounds. This
speculative scheme also involves inverse results for Gowers uniformity norms,
and is commented on briefly at the end of the paper.

\medskip
I am indebted to Diogo Oliveira e Silva for useful comments on the exposition.

\section{Results}
Let $\{\ell_j:\reals^d\to\reals^{d_j}\}$ be a finite collection of surjective
linear mappings.
For any Lebesgue measurable functions $g_j$ which are finite almost everywhere,
for any $\eps>0$, and for any compact subset $B\subset\reals^d$
consider the {\em sublevel sets}
\begin{equation} \label{sublevelsetsdefn}
E_\eps(P,g_1,\cdots,g_n)= \big\{
y\in B:
|P(y)-\sum_{j=1}^n g_j(\ell_j(y))|<\eps
\big\}.
\end{equation}
If a real-valued measurable function $P$ satisfies \eqref{decay},
then there is an upper bound for the
measures of these sublevel sets, of the form
\begin{equation} \label{ineq:sublevel}
\big| E_\eps(P,g_1,\cdots,g_n) \big| \le C\eps^\delta
\ \text{\em uniformly for all measurable functions $g_j$.}
\end{equation}
If instead $P$ satisfies \eqref{slowerdecay}, then there is 
a corresponding weakened version of \eqref{ineq:sublevel} in which
the right-hand side is replaced by a function of $\eps$ which tends
to zero as $\eps\to 0$.
Because of this connection with multilinear operators, we call
sets $E_\eps$ of the form \eqref{sublevelsetsdefn} multilinear sublevel sets. 

\eqref{ineq:sublevel}  
can be deduced from \eqref{decay} (with a smaller value of $\delta$ in \eqref{ineq:sublevel} in some cases).
To do so, fix a nonnegative cutoff function
$ h\in C^\infty_0(\reals)$ satisfying $h(t)=1$ whenever
$|t|\le 1$. Fix also $0\le\zeta\in C^\infty_0(\reals^m)$
such that $\zeta\equiv 1$ on $B$.
Then
\begin{multline*}
\big|\{x\in B: |P(x)-\sum_j g_j(\ell_j(x))|<\eps\}\big|
\le \int h[\eps\rp(P-\sum_j g_j\circ\ell_j)(x)]\,\zeta(x)\,dx
\\
= (2\pi)\rp \eps\int_\reals \widehat{h}(\eps\lambda)
\int e^{i\lambda(P(x)-\sum_j g_j(\ell_j(x))}\zeta(x)\,dx\,d\lambda.
\end{multline*}
Applying \eqref{decay} to the inner integral and continuing in 
a straightforward way leads to the sublevel set bounds.
\eqref{slowerdecay} leads in the same way to a corresponding variant of
\eqref{ineq:sublevel}. 

Our discussion relies on a different notion of degeneracy than
that defined above.
$f|_S$ will denote the restriction of a function $f$ to a set $S$.
\begin{definition}
Let $d,d_1,\cdots,d_n$ be arbitrary positive integers.
Let $P:\reals^d\to\complex$ be a polynomial,
and let $\ell_j:\reals^d\to \reals^{d_j}$ be
linear transformations for $1\le j\le n$.
$P$ is said to be {\em nondegenerate with a finite witness},
relative to $\{\ell_j\}$,
if there exists a finite set
$S\subset\reals^d$ 
such that $P|_S$ does not belong to the
span of the set of all functions $(f_j\circ\ell_j)|_S$.
\end{definition}

The union is taken over all indices $j$ and all functions $f_j$
before the span is formed.
An equivalent formulation is that there exist a finite set $S\subset\reals^d$ 
and scalars $c_s$ such that 
\begin{equation} \label{discrete1}
\sum_{s\in S} c_sP(s)\ne 0
\end{equation}
but 
\begin{equation} \label{discrete2}
\sum_{s\in S} c_s f_j(\ell_j(s))=0 \
\text{ for all indices $j$ and all functions $f_j$.}
\end{equation}
For generic finite sets $S$
the mapping $\ell_1(s)|_S$ is injective (unless $\ell_1\equiv 0$), 
whence no such scalars can exist.

The usefulness of discrete characterizations of nondegeneracy 
in the context of oscillatory integral theory
was recognized and exploited in \cite{ccw}.

From the theorem of Furstenberg and Katznelson we will deduce:

\begin{proposition}  \label{prop:sublevel1}
Suppose that a real-valued polynomial $P$
is nondegenerate with a finite witness, with respect to a finite collection
of surjective linear transformations
$\ell_j:\reals^d\to\reals^{d_j}$.
Then there exists a function $\Theta$ satisfying
\begin{equation} \label{Thetaproperty}
\lim_{\eps\to 0^+}\Theta(\eps)=0
\end{equation}
such that for every $\eps>0$ and any measurable functions $f_j$,
\begin{equation} \label{sublevelsetsbound}
|E_\eps(P,f_1,\cdots,f_n)|\le \Theta(\eps).
\end{equation}
\end{proposition}

In discussing the relation between these two notions
of nondegeneracy, we will employ the following auxiliary concept. 
A related, though distinct, concept was shown in \cite{trilinear} to be natural
in the context of a different question about multilinear operators.

\begin{definition}
Let $\ell_j:\reals^d\to\reals^{d_j}$ be finitely many linear transformations.
The collection $\{\ell_j\}$ is said to be {\em rationally commensurate}
if there exist invertible $\reals$-linear transformations $A:\reals^d\to\reals^d$
and $A_j:\reals^{d_j}\to\reals^{d_j}$
such that with respect to the standard bases of $\reals^d$ and of $\reals^{d_j}$,
the linear transformations $\tilde\ell_j = A_j^{-1}\circ\ell_j\circ A$
are all represented by matrices with integer entries.
\end{definition}

\begin{remark} \label{remark:ratcom}
It is easy to see that in the rationally commensurate
case, if $P$ fails to be nondegenerate with a finite witness, then there can be no sublevel
set bound of the form \eqref{sublevelsetsbound},\eqref{Thetaproperty}.
Indeed, we may change variables to arrange 
that each $\ell_j$ maps $\integers^d$ to $\integers^{d_j}$.
Let a bounded set $B\subset\reals^d$ and 
$\eps>0$ be given, and 
choose $r=r(\eps)>0$ so that $|P(x)-P(y)|<\eps$ whenever $|x-y|<r/2$.
Fix $\rho>0$ such that $|\ell_j(z)|\le \rho|z|$ for all $j$
and all $z\in\reals^d$.
Consider the lattice $r\integers^d=\{rn: n\in\integers^d\}$. For
each index $j$, $\ell_j(r\integers^d)\subset\reals^{d_j}$ is again a lattice. 
Let $\scriptl_r=r\integers^d\cap B$.

By assumption, there exist functions $f_j$ such that $P(y)=\sum_j f_j(\ell_j(y))$
for all $y\in\scriptl_r$. 
Since $\ell_j(\scriptl_r)\subset r\integers^{d_j}$,
there exists a constant $c_0>0$, independent of $r$, such that
for any $z\ne z'\in\scriptl_r$ and each index $j$,
either $|\ell_j(z)-\ell_j(z')|\ge c_0 r$,
or $\ell_j(z)=\ell_j(z')$. 

Only the values of $f_j$ on $\ell_j(\scriptl_r)$ come into play. 
Redefine $f_j$ so $f(x)\equiv f_j(\ell_j(y))$ for all $x$ in
the ball $B(y,cr)\subset\reals^{d_j}$ 
of radius $cr$ centered at each point $y\in \ell_j(\scriptl_r)$,
where $c$ is a positive constant, independent of $r$,
sufficiently small to ensure that these balls are pairwise disjoint
for distinct values of $y$.
 
The identity $P(y)=\sum_j f_j(\ell_j(y))$ still holds at each
point of $\scriptl_r$ for these modified functions $f_j$.
Moreover, if $x\in\reals^d$ and $|x-y|\le c'r$ for some $y\in\scriptl_r$,
where $c'$ is another sufficiently small positive constant independent of $r$,
then $|P(x)-\sum_j f_j\circ\ell_j(x)|
< \eps + |P(y)-\sum_j f_j\circ\ell_j(x)|$.
By construction, $f_j\circ\ell_j(x)=f_j\circ\ell_j(y)$.
Thus $|P(x)-\sum_j f_j\circ\ell_j(x)|<\eps$ whenever the distance
from $x$ to $\scriptl_r$ is $<c'r$. The measure of the set of all such
points $x\in B$ does not tend to zero as $\eps\to 0$,
contradicting \eqref{sublevelsetsbound}.
\qed
\end{remark}

Finitely witnessed nondegeneracy clearly implies nondegeneracy. Although we do not
know whether the converse holds in general, it is true 
in the rational case, which is one of the two main results of this paper:
\begin{theorem} \label{thm:discretenondegen}
Let $P:\reals^d\to\complex$ be a polynomial,
and let $\ell_j:\reals^d\to\reals^{d_j}$ be a finite collection of surjective linear 
transformations. If $\{\ell_j\}$
is rationally commensurate, and if $P$ is nondegenerate
relative to $\{\ell_j\}$,
then $P$ is nondegenerate with a finite witness relative to $\{\ell_j\}$.
\end{theorem}
\noindent
Most of the work in this paper is devoted to proving this purely algebraic fact.
Theorem~\ref{thm:discretenondegen} implies
Remark~\ref{remark:ratcom} in a stronger form,
for degeneracy of $P$ means that
there exist functions $f_j$ for which $P-\sum_j f_j\circ\ell_j \equiv 0$, and then
$E_\eps=B$ for all $\eps>0$.

Proposition~\ref{prop:sublevel1} and Theorem~\ref{thm:discretenondegen} together
yield our other main result.
\begin{theorem}  \label{thm:sublevel2}
Let a polynomial $P$ be nondegenerate with respect to a finite rationally commensurate 
collection of surjective linear transformations.
Then there exists a function $\Theta$ satisfying
\begin{equation} \label{eq:thetatozero}
\lim_{\eps\to 0^+}\Theta(\eps)=0
\end{equation}
such that for every $\eps>0$ and all measurable functions $f_j$,
\begin{equation} \label{eq:mainbound}
|E_\eps(P,f_1,\cdots,f_n)|\le \Theta(\eps).
\end{equation}
\end{theorem}
It deserves emphasis that oscillatory integral bounds of the type \eqref{decay} 
which imply this conclusion 
were proved in \cite{multiosc} in several cases, without any hypothesis
of rational commensurability. 
Natural questions which remain are whether the commensurability hypothesis
is superfluous, and whether $\Theta$ may always be taken to be of
power law form $C\eps^\delta$. 
The proof here certainly does not give power law bounds, since
it relies on a result of Szemer\'edi type. 

The proof of Theorem~\ref{thm:sublevel2}
is sufficiently robust to yield also the following variant.
A corresponding extension of one of the results of \cite{multiosc}
was established by Greenblatt \cite{greenblattsmooth}.
\begin{theorem} \label{thm:smoothcase}
Let $P$ be a $C^\infty$ real-valued function defined in
a neighborhood of $x_0\in\reals^d$.
Let $\{\ell_j\}$ be a rationally commensurate 
finite collection of surjective linear transformations 
$\ell_j:\reals^d\to\reals^{d_j}$.
Suppose that some Taylor polynomial for $P$ at $x_0$ is
nondegenerate with respect to $\{\ell_j\}$.
Then there exist a neighborhood $U$ of $x_0$
and a function $\Theta$ satisfying
$\lim_{\eps\to 0^+}\Theta(\eps)=0$
such that for every $\eps>0$ and all measurable functions $f_j$,
\begin{equation}
\big|\big\{x\in U: |(P-\sum_j f_j\circ\ell_j)(x)|<\eps\big\}\big|
\le \Theta(\eps).
\end{equation}
\end{theorem}

Another extension concerns periodic sublevel sets, in which 
$P-\sum_j f_j\circ\ell_j$ is viewed as taking values in $\reals/2\pi\integers$,
rather than in $\reals$. 
For simplicity, let $P$ be a polynomial. 
Define
\[
\norm{y} = \distance(y,2\pi\integers)
\]
for $y\in\reals$.
\begin{equation} 
E^\dagger_{\eps,\lambda}(P,f_1,\cdots,f_n) =
\{x\in B: \norm{\lambda P(x)-\sum_j f_j(\ell_j(x))}<\eps\}.
\end{equation}
Let us assume that the cutoff function $\eta$ appearing in 
\eqref{oscintegraldefn} is nonnegative,
and write $|E|=\int_E \eta$ for any measurable set $E\subset\reals^d$.
A uniform bound for multilinear oscillatory integrals of the form
$|\scripti_\lambda(f_1,\cdots,f_n) \le  C|\lambda|^{-\delta}$
for some $\delta\in(0,1)$
implies uniform bounds of the form 
\begin{equation}
\label{powersublevelmodtwopi}
\big| |E^\dagger_{\eps,\lambda}|-c_0\eps \big|
\le C\eps^\delta|\lambda|^{-\delta},
\end{equation}
where $c_0 = \int_{\reals^d}\eta(x)\,dx$.
Similarly a uniform decay bound 
$|\scripti_\lambda(f_1,\cdots,f_n)| \le  \Theta(\lambda)$,
where $\Theta(\lambda)\to 0$ as $|\lambda|\to\infty$,
implies uniform bounds 
\begin{equation}
\label{sublevelmodtwopi}
\big| |E^\dagger_{\eps,\lambda}|-c_0\eps \big|
\le \theta(\eps,|\lambda|^{-1})
\end{equation}
where $\theta(s,t)\to 0$ as $\min(s,t)\to 0^+$.
Conversely, such inequalities
imply uniform decay bounds for oscillatory integrals.
Thus it is natural to seek suitable uniform upper bounds for 
$|E^\dagger_{\eps,\lambda}|$ for nondegenerate polynomial phases $P$.

\begin{theorem} \label{thm:periodicsublevelsets}
Suppose that a polynomial $P$ is nondegenerate, relative to a rationally
commensurate set $\{\ell_j\}$ of surjective linear mappings.
Let $B\subset\reals^d$ be a bounded set.
Then there exists a positive function $\Theta$
satisfying $\Theta(t)\to 0$ as $t\to 0^+$
such that for all measurable functions $f_j$ and
all $|\lambda|\ge 1$,
\begin{equation}
\big|\{x\in B: \distance\big(\lambda P(x)-\sum_j f_j(\ell_j(x)),2\pi\integers\big)<\eps\}\big|
\le \Theta(\eps).
\end{equation}
\end{theorem}

In short,
$|E^\dagger_{\eps,\lambda}|\le \Theta(\eps)$.
This is of course weaker than the bound \eqref{sublevelmodtwopi},
so falls short of establishing the conjectured bound for multilinear oscillatory integrals.

\section{Proof of Proposition~\ref{prop:sublevel1}}

\begin{proposition} \label{prop:continuumfurstkatz}
Let $B\subset\reals^d$ be a bounded region,
and let $S\subset\reals^d$ be  a finite set
which contains $0$.
There exists a positive function $\Theta$ satisfying $\Theta(r)
\to 0$ as $r\to 0^+$, depending only on $S$ and on $B$,
with the following property: 
For any Lebesgue measurable set $E\subset B$ 
and any $r>0$, either 
(i) there exist $x\in B$ and $t\ge r$ such that $x+tS\subset E$, 
or (ii) $|E|\le\Theta(r)$.
\end{proposition}

\begin{proof}
Denote by $\#(A)$ the cardinality of a set $A$.
According to a theorem of Furstenberg and Katznelson \cite{furstkatz},
for any finite set $S\subset\integers^d$ 
there exists a positive function $\theta$, satisfying $\theta(N)\to 0$
as $N\to\infty$, such that for any set $A\subset \{1,2,\cdots,N\}^d$,
either there exist $0\ne n\in\integers$ and $x\in \integers^d$ such that
$x+nS\subset A$, or $\#(A) \le \theta(N)N^d$.

Proposition~\ref{prop:continuumfurstkatz} follows rather directly from this result. 
Under the additional assumption that the set $S$ in the hypothesis 
is contained in $\integers^d$, the reduction goes as follows:
Let $N$ be a large positive integer chosen so that $\tfrac12 r\le N\rp<r$.
Define $\Omega = \{\omega=(\omega_1,\cdots,\omega_d)\in\reals^d: 0\le\omega_j<N\rp
\text{ for all } 1\le j\le d\}$,
and define $\scriptl_{N,\omega} = N\rp\integers^d+\omega 
= \{N\rp n+\omega: n\in\integers^d\}$.
Let $E\subset B$, and 
suppose that conclusion (i) of Proposition~\ref{prop:continuumfurstkatz} 
fails to hold.  Decompose $E = \cup_{\omega\in\Omega} E_\omega$ 
where $E_\omega = E\cap \scriptl_{N,\omega}$. 
Then for any nonzero integer $j$ and point $x\in\scriptl_{N,\omega}$,
the set $x+jN^{-1}S$ is not contained in $E_\omega$.
Applying the theorem of Furstenberg and Katznelson
to $\tilde E_\omega = \{Ny: y\in E_\omega\}$ yields the bound
$\#(E_\omega) \le N^d\theta(N)$. Consequently
\begin{equation}
|E| = \int_{\Omega} \#(E_\omega)\,d\omega\le N^{-d} \sup_{\omega} \#(E_\omega) \le \theta(N),
\end{equation}
establishing Proposition~\ref{prop:continuumfurstkatz} under the auxiliary hypothesis.

The general case of Proposition~\ref{prop:continuumfurstkatz} follows 
from a particular case of the theorem of Furstenberg and Katznelson
by the following lifting argument.
Introduce $\reals^M=\reals^d\times\reals^S$ with coordinates $(x,t)$, 
where $t=(t_s: s\in S)\in\reals^S$.
Let $e_s\in\reals^{S}$ be the unit vector corresponding to the $s$-th coordinate.
Define $E^\dagger=E\times\reals^{S}$.

Introduce the shear transformation $T:\reals^M\to\reals^M$ defined by 
\begin{equation}
T(x,t) = (x-\sum_{\sigma\in S} t_\sigma \sigma,t), 
\end{equation}
and let $E^\ddagger = T(E^\dagger)$.
Then for any $r>0$, $t\in\reals^S$, $s\in S$, and $x\in\reals^d$,
\begin{equation} \label{liftingidentity}
x+ rs \in E
\text{ if and only if } 
T(x,t)+(0, r e_s)\in E^\ddagger.
\end{equation}
Indeed, $x+r s\in E$ is equivalent to $(x+ r s,t+r e_{s}) \in E^\dagger$.
Next
\begin{multline*}
T(x+ r s,
t+ re_{s}) 
= (x+ rs
- \sum_{\sigma\in S} t_\sigma\sigma - rs,
t+ re_{s}) 
\\
= (x-\sum_{\sigma} t_\sigma \sigma,
t+ re_{s} )
= T(x,t)+(0, re_{s}), 
\end{multline*}
whence
\begin{multline*}
x+ rs \in E
\Leftrightarrow
(x+ rs,
t+ re_{s})\in E^\dagger
\\
\Leftrightarrow T(x+ rs,
t+ re_{s} )\in E^\ddagger
\Leftrightarrow T(x,t)+(0, re_{s})\in E^\ddagger.
\end{multline*}

Let $S^\ddagger = \{(0,e_s): s\in S\}\subset\reals^M$.
Suppose now that $E$ satisfies the restriction
that $x+rS\subset E$ implies
$r\le\eps$. Then $E^\ddagger$ satisfies
a corresponding restriction: if $z\in \reals^M$ and
if $z+rS^\ddagger \subset E^\ddagger$,
then $r\le\eps$.
Indeed, there exists a unique point $(x,t)$ satisfying $T(x,t)=z$.
By \eqref{liftingidentity}, $z+rS^\ddagger\subset E^\ddagger$ 
if and only if $x+rS\subset E$.

For almost every $x\in\reals^d$,
we now have a set $E^* = \{(w,t)\in E^\ddagger: w=x \text{ and } |t|\le 1\}$,
contained in a fixed bounded subset of $\reals^{S}$, 
such that for any $z\in \reals^{S}$,
if $z+re_{s} \in E^*$ for every $s\in S$ then $r\le\eps$.
As was shown above,
this forces $|E^*\cap B'|\le \Theta(\eps)$ for any fixed bounded set $B'$. 
Therefore $|E|\le C\Theta(\eps)$ by Fubini's theorem.
\end{proof}

\begin{proof}[Proof of Proposition~\ref{prop:sublevel1}]

Suppose that $P$ is nondegenerate with a finite witness.
Fix a finite set $S$ and scalars $\{c_s: s\in S\}$
such that $\sum_{s\in S} c_s F(s)=0$ whenever
$F$ takes the form $\sum_j f_j\circ\ell_j$, but $\sum_{s\in S} c_s P(s)=1$.
In Proposition~?? it was convenient to assume that $0\in S$; here
if $0$ is not already an element of $S$, we may adjoin it, setting  $c_0=0$.

Let $f_j$ be arbitrary measurable functions;
for convenience we assume that $f_j$ is defined on all of $\reals^1$.
Set 
\begin{gather*}
h(y,r) = \sum_{s\in S} c_s 
\big(P(y+rs)-\sum_j (f_j\circ\ell_j)(y+rs)\big) 
\\
E_\eps=\{y\in B: |\big(P-\sum_j f_j\circ\ell_j \big)(y)|<\eps\}.
\end{gather*}
Then
\[
h(y,r)
=\sum_{s\in S} c_s 
\big(P(y+rs)-\sum_j (f_j\circ\ell_j)(y+rs)\big) 
\equiv \sum_{s\in S} c_s P(y+rs)
\]
is a polynomial function of $(y,r)\in\reals^d\times\reals$.
The set $S$ and coefficients $c_s$ were constructed in part to ensure that
this polynomial does not vanish identically.
Hence, by an elementary argument which is left to the reader,
there exist $A<\infty$, $\delta>0$, and $C<\infty$ such 
for any sufficiently small $\rho>0$, $B\subset\reals^d$
may be partitioned into the union of $O(\rho^{-d})$ 
dyadic cubes $Q_j$ of sidelength $\rho$, together
with a remainder set $B\setminus \cup_j Q_j$, in such a way that
(i) $|B\setminus \cup_j Q_j|\le C\rho^\delta$
and (ii) for each $j$, each $x\in Q_j$, and each $r\in(0,\rho]$,
$|h(x,r)|\ge  r^A$.

Choose $\rho=\eps^{1/2A}$.
If $x\in B$ and $x+rS\subset E_\eps$ then 
\begin{equation}
|h(x,r)|\le \sum_{s\in S}c_s 
\big|\big(P-\sum_j f_j\circ\ell_j\big)
(x+rs)\big| \le C\sum_{s\in S} \eps,
\end{equation}
which implies that
$r^A\lesssim \#(S)\eps$ 
if $x\in \cup_j Q_j$ and $r\le\rho$,
where $\#(S)$ denotes the cardinality of $S$.
Therefore
\begin{equation}
|E_\eps \cap Q_j| \le |Q_j| \Theta(C\eps^{1/A}/\rho),
\end{equation}
by Proposition~\ref{prop:continuumfurstkatz} applied to a dilate of $Q_j$.
Here the $C$ depends on the cardinality of $S$, which is a constant
in this context.

Summing over $j$ yields
\begin{equation}
|E_\eps| 
\le |B\setminus\cup_j Q_j| + \sum_j |E_\eps\cap Q_j| 
\le C\eps^{\delta/2A} + |B|\Theta(C\eps^{1/2A}),
\end{equation}
which is a bound of the desired form.
\end{proof}

\section{Proof of Theorem~\ref{thm:discretenondegen}}

Even if $P$ is nondegenerate, the restriction of $P$ to a generic
finite set $S$ will be degenerate relative to $\{\ell_j\}$.
Indeed, if the restriction of some $\ell_i$ to $S$
is injective, then any function on $S$ takes the form
$f_i\circ\ell_i$. Thus $S$ is a more promising candidate
to be a witness, if all of the mappings $\ell_j$ are far
from being injective on $S$. This motivates the use of finite
lattices as witnesses; the hypothesis of rational commensurability
will ensure a strong failure of injectivity for suitable lattices.

$M\integers^d$ will denote the set of all $(x_1,\dots,x_d)\in\integers^d$
for which each coordinate is divisible by $M$.

Recall that any finitely generated torsion-free $\integers$-module $\scriptm$ is isomorphic
to $\integers^n$ for some unique $n$; $n$ is called the rank of $\scriptm$.
Any submodule of $\integers^n$ is finitely
generated and torsion-free.
By the {\em rank} of a homomorphism of $\integers$-modules,
we mean the rank of its range; only finitely
generated and torsion-free ranges will arise in this paper.
Let $\scriptm\subset\integers^n$ be a sub-$\integers$-module of
rank $r$, and choose elements $e_1,\cdots,e_r\in\scriptm$
such that the mapping $(x_1,\cdots,x_r)
\mapsto x\cdot e = x_1e_1+\dots+x_re_r$ defines a bijection
of $\integers^r$ onto $\scriptm$.
If $q:\scriptm\to\complex$ is a polynomial, in the
sense that $q$ can be represented as a finite linear
combination over $\complex$ of the monomials $x\cdot e\mapsto
x_1^{\gamma_1}\cdots x_n^{\gamma_n}$ with each exponent $\gamma_j$
a nonnegative integer, then such a representation is unique.

The analysis will make use of difference operators. For any vector $y$, $D_y$
denotes the operator $D_y f(x) = f(x+y)-f(x)$.
These operators all commute with one another.
If $L$ is a linear mapping then
\begin{equation} \label{pushforward}
D_y(f\circ L) = \big(D_{L(y)}(f)\big)\circ L.
\end{equation}
A version of Leibniz's rule is
\begin{equation} \label{leibniz}
D_v(fg) = D_v(f)\cdot g + f_v\cdot D_v(g) \text{ where } f_v(x) = f(x+v).
\end{equation}

We will need the following elementary property of polynomials,
whose proof is omitted.
\begin{lemma} \label{lemma:nullsatz}
For any $d,r$ there exists $N<\infty$ such that for any
polynomial $P:\reals^d\to\complex$ of degree $\le r$,
if $P(x)\equiv 0$
for all $x\in \integers^d$ satisfying $|x|\le N$,
then $P(x)=0$ for all $x\in\reals^d$.
\end{lemma}

The next lemma describes solutions of certain difference equations.
\begin{lemma} \label{lemma:makepolys}
Let $(v_j)\subset\integers^n$ be any finite list of
nonzero vectors, not necessarily distinct, and let
$\scriptd$ be the difference operator $\scriptd=\prod_j D_{v_j}$.
Then there exist $C,r<\infty$,
a positive integer $M$,
and finitely many $\integers$-module homomorphisms
$\ell_\gamma:\integers^n\to\integers^{n_\gamma}$ where $n_\gamma<n$,
such that for any sufficiently large $N<\infty$ and any function
$f:\integers^n\to\complex$
which satisfies $\scriptd(f)(x)=0$ for all $x\in\integers^n$
satisfying $|x|\le N$,
there exists a representation
\begin{equation} \label{eq:makepolysrep}
f(x) = \sum_\gamma q_\gamma(x)(h_\gamma\circ\ell_\gamma)(x)
\end{equation}
valid for all $x\in M\integers^n$ satisfying $|x|\le N-C$,
where the $q_\gamma:\integers^n\to\complex$ are polynomials
of degrees at most $r$,
and $h_\gamma:\integers^{n_\gamma}\to\complex$ are some functions.
\end{lemma}

\begin{proof}[Sketch of proof]
Proceed by induction the number of factors $D_{v_j}$. Thus
suppose it to be given $\scriptd D_w(f)=0$ vanishes for all
$x$ in the region indicated, where $w\ne 0$ and $\scriptd$
is as above. Applying the induction hypothesis gives
a representation
$D_w(f)(x) = \sum_\gamma q_\gamma(x)(h_\gamma\circ\ell_\gamma)(x)$
of the above form,
for $x\in M\integers^n$ satisfying $|x|\le N-C$.

It is awkward to proceed further, because $w$ need not lie in $M\integers^n$.
However,
$D_w\scriptd(f)=\scriptd D_w(f)$ vanishes for $|x|\le N$,
from which it follows that $\scriptd D_{Mw} (f)=D_{Mw}\scriptd(f)$
also vanishes for $|x|\le N-C(|w|)$; note that $M$ depends only
on $\{v_j\}$, not on $f$. Thus we may suppose from the outset
that $w\in M\integers^n$.

It is straightforward to solve the equation
$D_w(F)(x) = \sum_\gamma q_\gamma(x)(h_\gamma\circ\ell_\gamma)(x)$
with a solution $F$ in the desired form \eqref{eq:makepolysrep},
in the region $M\integers^n\cap\{x: |x|\le N-C\}$,
with the initial condition that $F$ vanishes on a suitable
submodule of rank $n=1$ which does not contain $w$.
This equation is solved term-by-term, distinguishing the terms
for which $\ell_\gamma(w)=0$ from those for which $\ell_\gamma(w)\ne 0$.

Finally since $D_w(f-F)\equiv 0$ on an appropriate domain,
it must take the form $h\circ\ell$, where $\ell$ has rank $n-1$
and $\ell(w)=0$.
\end{proof}

\begin{lemma} \label{lemma:polynomialrep}
Let $A$ be a finite set of indices,
$L_\alpha:\integers^d\to\integers^{d_\alpha}$ be $\integers$--linear mappings,
and let $f_\alpha$ be arbitrary functions.
Let $P:\integers^d\to\complex$ be a polynomial which takes the form
\begin{equation} \label{givendecomp}
P(x) = \sum_{\alpha\in A} (f_\alpha\circ L_\alpha)(x)
\text{ for all $x\in\integers^d$ satisfying $|x|\le N$.}
\end{equation}
If $N$ is sufficiently large
then there exist polynomials $p_\alpha$ and positive integers $M,N^*$
such that
\begin{equation}\label{desireddecomp}
P(x) = \sum_{\alpha} \big(p_{\alpha} \circ L_\alpha\big)(x)
\end{equation}
for all $x\in M\integers^d$ satisfying $|x|\le N^*$.
$N^*\to\infty$ as $N\to\infty$,
while $M$ and the degrees of the $p_\alpha$
remain uniformly bounded, provided that
the linear mappings $L_\alpha$
and the degree of $P$ remain fixed.
\end{lemma}

The functions $f_\alpha$ in such a decomposition
$P = \sum_\alpha f_\alpha\circ L_\alpha$
are not necessarily polynomials. There are also cases in which they are necessarily
polynomials, but are not necessarily unique.

A related result was established in \cite{multiosc}: If
a polynomial $P$ admits a decomposition $P(x)=\sum_\alpha f_\alpha\circ L_\alpha$
on $\reals^d$, where the $f_\alpha$ are merely distributions, then
it admits such a decomposition with those distributions replaced by
polynomials. The simple proof given in \cite{multiosc} does not seem to
adapt directly to the present discrete setting.

By admissible data we mean the collection of mappings $L_\alpha$, and the degree of $P$.
It will be important, in both the proof and application of Lemma~\ref{lemma:polynomialrep},
that $N^*,M$ and the degrees of $p_\alpha$ depend only on admissible data.
We will say that a polynomial has
bounded degree if its degree is bounded above by a quantity which
depends only on admissible data.
Likewise, by a {\em large finite submodule} of $\integers^{n}$
we mean, in the context of Lemma~\ref{lemma:polynomialrep},
the set of all $x=(x_1,\cdots,x_{n})\in \integers^{n}$
such that $|x|\le N^\sharp$
and each coordinate $x_j$ is divisible by some $M^\sharp$,
where $N^\sharp\to\infty$ while $M^\sharp$ remains uniformly
bounded,
as the parameter $N$ given in the hypotheses tends to $\infty$.
$N^\sharp,M^\sharp$ are permitted to depend on admissible data.
This is of course an abuse of language, since these ``submodules''
are not closed under addition.

\begin{proof}[Conclusion of proof of Theorem~\ref{thm:discretenondegen}]
If $P$ is not nondegenerate with a finite witness, then
Lemma~\ref{lemma:polynomialrep}, applied to the polynomials $x\mapsto P(Nx)$,
asserts that for any open ball $B\subset\reals^d$ centered at the origin,
for each sufficiently large integer $N$, $P|_{B\cap MN^{-1}\integers^d}$
can be expressed in the form $\sum_j Q_j\circ \ell_j$,
where the polynomials $Q_j$ may potentially depend on $N$,
but have uniformly bounded degrees. By Lemma~\ref{lemma:nullsatz},
applied again to $x\mapsto P(MN^{-1}x)$ for a certain constant $M$,
this implies that $P-\sum_j Q_j\circ\ell_j$ vanishes
identically on $\reals^d$. Thus $P$ is degenerate
relative to $\{\ell_j\}$.
\end{proof}

\medskip
In the proof of Lemma~\ref{lemma:polynomialrep}, the quantity $M$ appearing in
its conclusion will repeatedly be replaced
by a larger multiple of itself throughout an inductive procedure.
All of these quantities will be denoted by the same symbol $M$,
with the understanding that $M$ is always bounded above by a quantity
which depends only on  admissible data.

\begin{proof}[Proof of Lemma~\ref{lemma:polynomialrep}]
Lemma~\ref{lemma:polynomialrep} will be proved by an inductive scheme which
involves more general representations of $P$.
To set this up, suppose that
$P:\integers^d\to\complex$ is a polynomial which takes the form
\begin{equation} \label{givendecomp2}
P(x) = \sum_{\alpha\in A}\sum_j (Q_{\alpha,j}\circ L_\alpha)
\cdot(h_{\alpha,j}\circ\ell_{\alpha,j})(x)
\end{equation}
for all $x\in\integers^d$ satisfying $|x|\le N$.
Here $A$ is a finite set of indices,
$j$ ranges over a finite set of indices for each $\alpha\in A$,
$L_\alpha,\ell_{\alpha,j}$ are $\integers$--linear mappings from $\integers^d$ to
some $\integers^{n(\alpha)}$ and $\integers^{n(\alpha,j)}$, respectively,
\[
\nullspace(L_\alpha)\subset\nullspace(\ell_{\alpha,j}),
\]
$h_{\alpha,j}$ are arbitrary functions, and
$Q_{\alpha,j}$ are polynomials with domains $\integers^{n(\alpha,j)}$.
Admissible data are now the collection of
mappings $L_\alpha,\ell_{\alpha,j}$, and the degrees of $P,Q_{\alpha,j}$.
Suppose finally that there exists at least one pair $(\alpha,j)$ for which
$\ell_{\alpha,j}$ has positive rank; otherwise \eqref{givendecomp2} would
already be a representation of the desired form.

As the conclusion of the inductive step,
we claim that if $N$ is sufficiently large then there exist
polynomials $Q_{\alpha,k}^*$, linear mappings $\ell_{\alpha,k}^*$,
functions $h_{\alpha,k}^*$, and integers $M,N^*$,
such that
\begin{equation}\label{desireddecomp2}
P(x) = \sum_{\alpha\in A}\sum_k (Q_{\alpha,k}^* \circ L_\alpha)
\cdot (h_{\alpha,k}^*\circ\ell_{\alpha,k}^*)(x)
\end{equation}
for all $x\in M\integers^d$ satisfying $|x|\le N^*$.
Here $A$ is the same index set as in \eqref{givendecomp2},
$k$ ranges over a finite index set for each $\alpha\in A$,
and
$\nullspace(L_\alpha)\subset\nullspace(\ell_{\alpha,k}^*)$.
The index sets over which $k$ ranges need not coincide with those
over which $j$ ranges in \eqref{givendecomp2}, and in practice
will be larger.
Moreover as $N\to\infty$, $N^*\to\infty$ while
$M$ and the degrees of the $Q_{\alpha,k}^*$ remain uniformly bounded,
provided that the admissible data remain fixed.
Finally, and essentially,
we claim that there exists such a representation
\eqref{desireddecomp2} which is simpler
than the given one \eqref{givendecomp2}, in the sense that either
the maximum (over all pairs $\alpha,k$) of the ranks of the $\ell_{\alpha,k}^*$
is strictly less than the maximum rank
of all $\ell_{\alpha,j}$, or that the two maxima are equal and that the
number of index pairs $(\alpha,k)$ for which
$\ell_{\alpha,k}^*$ has maximal rank is strictly less than
the number of pairs $(\alpha,j)$ for which
$\ell_{\alpha,j}$ has maximal rank.

At each stage of the induction,
a hypothesis on $\integers^d$ leads to a conclusion only on some $M\integers^d$,
but then $M\integers^d$ can be reidentified with $\integers^d$ in the following step.
Finitely many induction steps bring us to the situation in which
every $\ell_{\alpha,j}$ has rank zero.
In that case, each $h_{\alpha,j}\circ\ell_{\alpha,j}$
is a constant. For each $\alpha$, $\sum_j (Q_{\alpha,j}\circ L_\alpha)\cdot
(h_{\alpha,j}\circ\ell_{\alpha,j})$ is the composition of a single polynomial with $L_\alpha$,
so \eqref{givendecomp2} has the desired form and the proof would be complete.
Thus in order to prove the lemma, it suffices to carry out this induction step.

Supposing that some $\ell_{\alpha,j}$ has nonzero rank,
choose some $(\alpha_0,j_0)$ such that $\ell_{\alpha_0,j_0}$ has maximal rank.
Equivalently, $\nullspace(\ell_{\alpha_0,j_0})\subset\rationals^d$
has minimal dimension among all such nullspaces, as vector spaces over $\rationals$.

Let $(\beta,k)$ be an arbitrary order pair of indices. If
\begin{equation} \label{betacondition}
\nullspace(\ell_{\beta,k})\ne\nullspace(\ell_{\alpha_0,j_0})
\end{equation}
then by minimality,
$\nullspace(\ell_{\beta,k})$ is not a subset of $\nullspace(\ell_{\alpha_0,j_0})$.
For each such pair $(\beta,k)$ choose some vector
\[y_{\beta,k}\in \integers^d\cap \nullspace(\ell_{\beta,k})
\setminus\nullspace(\ell_{\alpha_0,j_0}).\]
Define
\begin{equation} \label{scriptddefn}
\scriptd = \prod_{(\beta,k)}D_{y_{\beta,k}}
\end{equation}
where the composition product is taken over all ordered pairs of indices $(\beta,k)$ satisfying
\eqref{betacondition}.
Then
\[
\scriptd(f\circ\ell_{\beta,k})\equiv 0 \text{ on } \integers^d
\]
for all $(\beta,k)$ satisfying \eqref{betacondition},  for all functions $f$.
Indeed, $D_{y_\beta,k}$ annihilates
all such functions, and the factors in \eqref{scriptddefn} all commute.

Consequently there exists a positive integer $a$ such that
\begin{equation}
\scriptd^a\big((Q_{\beta,k}\circ L_\beta)\cdot(f\circ\ell_{\beta,k})\big)
\equiv 0
\end{equation}
for all $(\beta,k)$ satisfying \eqref{betacondition};
$a$ may be chosen to depend only on the degrees of the given polynomials
$Q_{\alpha,j}$. This follows from Leibniz's rule \eqref{leibniz}.
For $D_v^b$ annihilates any polynomial of degree strictly less than $b$,
for any vector $v$, and $Q_{\alpha,j}\circ L_\alpha$ is a polynomial on $\integers^d$
whose degree does not exceed that of $Q_{\alpha,j}$.

We may also choose $a$ sufficiently large to ensure that
$\scriptd^a(P)\equiv 0$. Therefore
\begin{equation} \label{reducedvanishing}
\sum'_{\alpha,j} \scriptd^a\big((Q_{\alpha,j}\circ L_\alpha)
\cdot(h_{\alpha,j}\circ\ell_{\alpha,j}) \big)
\equiv 0,
\end{equation}
where the notation $\sum'_{\alpha,j}$ indicates that the sum is taken over all
pairs $(\alpha,j)$ which satisfy
\begin{equation} \label{alphacondition}
\nullspace(\ell_{\alpha,j}) = \nullspace(\ell_{\alpha_0,j_0}).
\end{equation}
\eqref{reducedvanishing} holds at all points $x\in\integers^d$
which satisfy $|x|\le N-C$, where $C$ depends only on the vectors
$y_\beta$ and the exponent $a$, thus only on permissible quantities;
in particular, $C$ is independent of $N$.

Define $U\subset\integers^d$ to be the nullspace of $\ell_{\alpha_0,j_0}$.
Choose a sub-$\integers$-module $V\subset\integers^d$ which is complementary to $U$
in the sense that
$U,V$ are linearly independent over $\rationals$, and $U\cup V$ spans
$\rationals^d$ over $\rationals$.
Then $U+V$ contains $M\integers^d$, for some positive integer $M$ which
depends only on $\ell_{\alpha_0,j_0}$.
We write $(u,v)$ to denote an arbitrary point $(u,v)=u+v= (u,0)+(0,v)\in U+V$.

For each pair $(\alpha,j)$ satisfying \eqref{alphacondition},
$(h_{\alpha,j}\circ\ell_{\alpha,j})(u,v)
\equiv (h_{\alpha,j}\circ\ell_{\alpha,j})(0,v)$.
The factor $(Q_{\alpha,j}\circ L_\alpha)(u,v)$ potentially depends
on both variables, but any dependence on $v$ can be incorporated
into $h_{\alpha,j}$ since $\ell_{\alpha,j}$ is injective on $V$
for all pairs $(\alpha,j)$ satisfying \eqref{alphacondition}.
Thus there are representations
\[
(Q_{\alpha,j}\circ L_\alpha)(u,v) \cdot
(h_{\alpha,j}\circ\ell_{\alpha,j})(u,v)
= \sum_r
(\tilde Q_{\alpha,j,r}\circ L_\alpha)(u,0) \cdot
(\tilde h_{\alpha,j,r}\circ\ell_{\alpha,j})(0,v)
\]
where $r$ runs over a finite index set which depends on $(\alpha,j)$, 
and the $\tilde Q_{\alpha,j,r}$
are polynomials. Both the cardinalities of these index sets,
and the degrees of these polynomials, are bounded above by
quantities which depend only on admissible data.

Let $K$ be the maximum degree of
all the polynomials $\tilde Q_{\alpha,j,r}$.
Decompose $\tilde Q_{\alpha,j,r} = Q^\dagger_{\alpha,j,r}
+ R_{\alpha,j,r}$ where $Q^\dagger_{\alpha,j,r}$ is
homogeneous of degree $K$,
while the remainders $R_{\alpha,j,r}$ have degrees strictly less than $K$.
By Leibniz's rule \eqref{leibniz},
\begin{equation} \label{secondreducedvanishing}
0 \equiv
\sum'_{\alpha,j} \sum_{r} (Q^\dagger_{\alpha,j,r}\circ L_\alpha)(u,0)
\cdot
\scriptd^a\big((\tilde h_{\alpha,j,r}\circ\ell_{\alpha,j})(0,v)\big)
+ R(u,v)
\end{equation}
 for all $(u,v)\in U+V$ satisfying $|(u,v)|\le N-C$,
where $C<\infty$ depends only on admissible data,
and $R$ can be expressed as a polynomial in $u$ of degree
$\le K-1$, whose coefficients are functions of $v$.
Each term
$(Q^\dagger_{\alpha,j,r}\circ L_\alpha)(u,0)$ is a homogeneous polynomial
of degree $K$. 

Since the degrees of all polynomials in play here are bounded uniformly
in $N$, it follows from
\eqref{secondreducedvanishing} and Lemma~\ref{lemma:nullsatz} that
\begin{equation} \label{thirdreducedvanishing}
0 \equiv
\sum'_{\alpha,j} \sum_{r} (Q^\dagger_{\alpha,j,r}\circ L_\alpha)(u,0)
\cdot
\scriptd^a\big((\tilde h_{\alpha,j,r}\circ\ell_{\alpha,j})(0,v)\big),
\end{equation}
again for all $(u,v)\in U+V$ satisfying $|(u,v)|\le N-C$.

There are now two cases. In Case 1,
\begin{equation} \label{firstcase}
\scriptd^a\big((\tilde h_{\alpha,j,r}\circ\ell_{\alpha,j})(0,v)\big)\equiv 0
\end{equation}
for all $(0,v)\in V$ satisfying $|v|\le N-C$,
for each pair $(\alpha,j)$ satisfying \eqref{alphacondition}.
By \eqref{pushforward},
\[
\scriptd^a\big((\tilde h_{\alpha,j,r}\circ\ell_{\alpha,j})\big)
= (\scriptd'_{\alpha,j}(\tilde h_{\alpha,j,r})\big)\circ\ell_{\alpha,j}
\]
where
\[
\scriptd'_{\alpha,j} = \prod_{\beta,k} D_{\ell_{\alpha,j}(y_{\beta,k})}^a,
\]
with the product taken over all pairs $(\beta,k)$ satisfying \eqref{betacondition}.
Since $\ell_{\alpha,j}(y_{\beta,k})\ne 0$,
Lemma~\ref{lemma:makepolys} asserts that
for each such pair $(\alpha,j)$, the restriction
of $\tilde h_{\alpha,j}$ to $\ell_{\alpha,j}(V)$
can be decomposed as a finite sum of terms,
each of which is the product of a polynomial of uniformly bounded degree
with a function of the form
$h^\sharp\circ\ell^\sharp\circ\ell_{\alpha,j}$
for some $\integers$--linear mapping $\ell^\sharp$
whose rank is strictly less than the rank of $\ell_{\alpha,j}$,
and some function $h^\sharp$ whose domain is the range of $\ell^\sharp$.
This representation holds on the set of all $(0,v)\in MV$
satisfying $|v|\le cN-C$, where $M,c,C$ depend
only on admissible data.
Any polynomial composed with $\ell_{\alpha,j}$
can be rewritten as a polynomial composed with $L_\alpha$,
since $\nullspace(L_\alpha)\subset\nullspace(\ell_{\alpha,j})$
by hypothesis.
Since $M\integers^d\subset U+V$ for some positive integer $M$,
we have reduced matters to a situation which
satisfies the hypothesis of the induction step
on a large finite submodule of $\integers^d$.

Consider next Case 2, in which there exists at least one pair $(\alpha,j)$ for which
\eqref{firstcase} fails to hold.
Choose $\gamma,i,r$ and some $(0,v_1)\in V$ such that
$(\scriptd'_{\gamma,i}(\tilde h_{\gamma,i,r})\big)\circ\ell_{\gamma,i}(0,v_1)\ne 0$.
Specialize \eqref{thirdreducedvanishing} to $v=v_1$
and solve the resulting equation for $(Q^\dagger_{\gamma,i,r}\circ L_{\gamma})(u,0)$ as a
$\complex$-linear combination of the other $(Q^\dagger_{\alpha,j,s}\circ L_\alpha)(u,0)$.
The term
$(Q^\dagger_{\gamma,i,r}\circ L_{\gamma})(u,0)\cdot
\scriptd^a\big(\tilde h_{\gamma,i}\circ\ell_{\gamma}(0,v) \big)$
is thus expressed as a $\complex$-linear combination of hybrid terms
\[(Q^\dagger_{\alpha,j,s}\circ L_{\alpha})(u,0)\cdot
\scriptd^a\big(\tilde h_{\gamma,i,s}\circ\ell_{\gamma,i}(0,v) \big)\]
for all $(u,v)$ in a large finite submodule of $U+V$.
However, since $\ell_{\alpha,j}\big|_V$ is injective, each
$\tilde h_{\gamma,i,s}\circ\ell_{\gamma,i}(0,v)$
can be reexpressed in the form
$h^\flat_{\alpha,j,s}\circ\ell_{\alpha,j}(0,v)$, and thus
each hybrid term is reexpressed as
\[(Q^\dagger_{\alpha,j,s}\circ L_{\alpha})(u,0)\cdot
\scriptd^a\big(h^\flat_{\alpha,j,s}\circ\ell_{\alpha,j}(0,v) \big).\]

The result is that the degree of at least one of the polynomials
$\tilde Q_{\alpha,j,r}$ has been decreased, and the degrees of none have increased. 
This process can be iterated
until either  Case 1 eventually arises, or all $\tilde Q_{\alpha,j,r}$ have degree zero.
In the former event, the proof is complete by induction.

In the latter event, the sum 
$\sum'_{\alpha,j} \big(Q_{\alpha,j}\circ L_\alpha\big)
\cdot\big(h_{\alpha,j}\circ \ell_{\alpha,j}\big)$ 
which appeared in the initial representation of $P$ can be rewritten
more simply as a single term $h^\dagger_{\alpha_0}\circ\ell_{\alpha_0}$.
The relation \eqref{reducedvanishing} becomes simply
$\scriptd^a\big(h^\dagger_{\alpha_0}\circ\ell_{\alpha_0} \big)\equiv 0$.
This again makes Lemma~\ref{lemma:makepolys} applicable, so $h^\dagger_{\alpha_0}$
can be represented,
on a large finite submodule of its domain,
as a finite sum of products of
polynomials of bounded degrees multiplied by
functions composed with linear mappings of ranks strictly
less than the rank of $\ell_{\alpha_0}$, all of which factor
through $\ell_{\alpha_0}$, and hence through $L_{\alpha_0}$. 
Thus matters are again reduced to
a prior induction step, completing the proof.
\end{proof}

\begin{remark}
Suppose that $d_j=1$ for all $1\le j\le n$, 
so that $\ell_j:\reals^d\to\reals^1$, 
and that $\kernel(\ell_i)\ne\kernel(\ell_j)$ whenever $i\ne j$.
If $P:\reals^d\to\reals$ is a polynomial,
and if $P = \sum_j f_j\circ\ell_j$ for certain distributions $f_j$
defined in $\reals^1$, 
then necessarily each $f_j$ is a polynomial. Moreover, the degree of $f_j$
is majorized by a quantity depending only on $d,n$ and the degree of $P$.
This fact holds without the hypothesis of rational commensurability;
it can be proved by a variant of the reasoning used in the proof
of Theorem~\ref{thm:discretenondegen},
with difference operators replaced by differential operators
$\prod_{k\ne j} (v_k\cdot\nabla)^{a_k}$
where $\ell_k(v_k)=0$ for all $k\ne j$ but $\ell_j(v_j)\ne 0$,
and with the exponents $a_k$ chosen so that the degree of 
the operator exceeds the degree of $P$.

However, without the restriction $d_j=1$, the $f_j$ need not be polynomials; 
$0$ can be represented as $f_1(x_2,x_3)+f_2(x_1,x_3)+f_3(x_1,x_2)$ in many ways.
This difficulty is responsible for much of the complexity in the above proof.
\end{remark}

\section{Two extensions}

We discuss here the proofs for two extensions,
Theorems~\ref{thm:smoothcase} and Theorem~\ref{thm:periodicsublevelsets}. 
The former concerns $C^\infty$ phases which
are not necessarily polynomials, but satisfy a finite order nondegeneracy
condition. 

\begin{proof}[Proof of Theorem~\ref{thm:smoothcase}]
Let $P$ be a $C^\infty$ real-valued function satisfying the hypothesis,
and let a small $\eps>0$ be given.
Fix a large positive integer $N$, to be specified below.
Let $\rho>0$ be a function of $\eps$, to be determined.
Partition a neighborhood of $U$ into cubes $Q_k$, each of sidelength $\rho$.
For each $k$ let $P_k$ be the Taylor polynomial of degree $N-1$
for $P$ at the center point $c_k$ of $Q_k$.
Then 
\begin{equation}
|(P-P_k)(x)|\le C\rho^N \text{ for all $x\in Q_k$}.
\end{equation}

Define $L_k(y) = c_k+\rho y$, and
$\tilde P_k = P_k\circ L_k$. 
$\ell_{j}\circ L_k$ is now affine linear rather than linear,
but for any function $f_j$
we can write $f_j\circ\ell_j\circ L_k = \tilde f_j\circ\ell_{j}$
where $\tilde f_j$ is an appropriate translate and dilate of $f_j$,
depending on $j,k$.
If $N$ is sufficiently large then $\tilde P_k$ 
as a mapping whose domain is the unit cube centered at the origin,
is nondegenerate relative to the (affine) linear transformations $\ell_j\circ L_k$.
Our hypotheses do not guarantee that these polynomials are nondegenerate uniformly
in $k$ in any sense, but 
as in the proof of Proposition~\ref{prop:sublevel1}, 
if $\rho$ is chosen to be an appropriate positive power of $\eps$ then
the identity
$\sum_{s\in S} c_s (\tilde P_k - \sum_j f_j\circ\ell_j)(x+rs) = 
h_k(x,r) = h(L_k(x),\rho r)$
can be exploited 
to obtain a bound for most $k$, while the sum of the measures of the remaining cubes 
$Q_k$ is small.
\end{proof}

Theorem~\ref{thm:periodicsublevelsets} is a stronger result,
in which $P-\sum_j f_j\circ\ell_j$ is regarded as taking
values in the quotient space $\reals/2\pi\integers$. Sublevel sets
are then typically larger, yet turn out to satisfy the same upper bounds.
Recall the notation $\norm{t} = \distance(t,2\pi\integers)$.
The following simple fact will be used in the proof of this theorem.

\begin{lemma} \label{lemma:corollaryofVdC}
There exists $C<\infty$ such that 
for any $C^2$ function $\phi:[0,1]\to\reals$ satisfying
$\phi'(t)\ge 1$ and $\phi''(t)\ge 0$ for all $t\in[0,1]$,
\[|\{t\in[0,1]: \norm{\phi(t)}\le\delta\}|\le C\delta\log(1/\delta)\]
for all $\delta\in (0,\tfrac12]$.
\end{lemma}

\begin{proof}
By van der Corput's lemma, $\int_0^1 e^{i\lambda\phi(t)}\,dt
= O(|\lambda|^{-1})$ as $|\lambda|\to\infty$, for $\lambda\in\reals$.
For any small $\delta>0$
there exists a nonnegative $2\pi$--periodic function $\eta$ which satisfies
$\eta(t)\ge 1$ whenever $\norm{t}\le\delta$,
and $|\widehat{\eta}(n)|\le 
C\delta (1+\delta|n|)^{-2}$ for 
all $n\in\integers$, with $C<\infty$ independent of $\delta$.
Then
\begin{align*}
\int_0^1 \eta(\phi(t))\,dt
&= \int_0^1 \sum_n \widehat{\eta}(n)e^{in\phi(t)}\,dt
\\
&\le \sum_{n\in\integers} 
C\delta (1+\delta|n|)^{-2}
(1+|n|)^{-1}
\\
&\le C\delta\log(1/\delta).
\end{align*}
\end{proof}

In the next lemma, $x\in\reals^d$, while $r\in\reals$.
\begin{lemma}
Let $h=h(x,r)$ be a polynomial $\sum_{k=0}^m p_k(x)r^k$
where the $p_k$ are polynomials in $x$, and $p_m$ is a nonzero
constant.
Fix a bounded ball $B$.
There exist
$c,C,A\in\reals^+$ such that for all $\delta\in(0,1]$
and all $\lambda\in\reals$ satisfying $|\lambda|\ge 1$,
\begin{equation} 
\label{eq:periodicscalarsublevel}
|\{(x,r)\in B\times(0,1]: \norm{\lambda h(x,r)}\le\delta r^A \}|
\le C\delta^c. 
\end{equation}
\end{lemma}
This can be deduced from the preceding lemma. The details are left to the reader.
 
\begin{proof}[Outline of proof of Theorem~\ref{thm:periodicsublevelsets}]
This follows from a small modification of the arguments already indicated.
If $P$ is nondegenerate then after a change of variables,
all the $\ell_j$ can be represented by matrices with integer entries,
and there exist a finite set $S\subset\integers^d$
and coefficients $c_s\in\integers$
satisfying \eqref{discrete1} and \eqref{discrete2}. 
Indeed, the construction already given yields a finite witness
set $S\subset\integers^d$.
For such a set, the vector space of all  $(c_s)_{s\in S}\in \reals^{|S|}$
satisfying \eqref{discrete2} is the null space of a certain
matrix with integer entries, hence is spanned over $\reals$
by elements of $\integers^{S}$.

By taking the coefficients $c_s$ to be integers and repeating the
above reasoning as in the above discussion of sublevel sets, we conclude that
$E^\dagger_{\eps,\lambda}$
can contain no finite point configuration
$x+rS$
for which the pair $(x,r)$ satisfies 
\begin{equation} \label{modifiedexclusion}
\distance(\lambda h(x,r),2\pi \integers)\ge C\eps, 
\end{equation}
where $h(x,r) = \sum_{s\in S} c_s P(x+rs)$ 
is a polynomial in $r$ of positive degree, whose coefficients
are polynomial functions of $x$ and whose leading coefficient
is independent of $x$.

The set of all $(x,r)$ satisfying \eqref{modifiedexclusion}
is more complicated than the corresponding set in the proof
of Proposition~\ref{prop:sublevel1}, so some additional 
preparation is needed before the theorem of Furstenberg
and Katznelson can be applied.
Fix any bounded set $B\subset\reals^d$.

A short calculation using 
\eqref{eq:periodicscalarsublevel}
with the substitutions $r=tj/N$
and $\delta r^A = N^{-A-A'}$ shows that
there exist $A<\infty$  and $c>0$ such that
whenever $|\lambda|\ge 1$,
for any $N\ge 1$ and 
any $j\in\{1,2,\cdots,N\}$,
\[
|\{(x,t)\in B\times[\tfrac12,1]: \norm{\lambda h(x,tj/N)} \le N^{-A'-A} \}|
\le CN^{1-cA'}.
\]
Choose $A'$ so that $N^{2-cA'}\equiv N^{-1}$.
By applying Fubini's theorem and taking unions of exceptional sets
over all the $N$ parameters $j\in\{1,2,\cdots,N\}$, 
we lose a factor of $N$ and hence conclude that
for any $|\lambda|\ge 1$,
there exists $t\in[\tfrac12,1]$
such that
\[
\norm{\lambda h(x,tj/N)} \ge cN^{-A-A'}
\text{ for all $j\in\{1,2,\cdots,N\}$ and
all $x\in B\setminus\scripte$,}
\]
where the exceptional set $\scripte$ satisfies 
\[ |\scripte|\le CN^{2-cA'}=CN^{-1}. \] 

In combination with \eqref{modifiedexclusion},
this permits the theorem of Furstenberg and Katznelson
to be applied, in the same spirit as in the proof of 
Proposition~\ref{prop:sublevel1} above.
\end{proof}



A strong variant, to the effect that
there is a uniform sublevel set estimate of this form with $\lambda=1$ for
all polynomials $P$ of bounded degree that are uniformly nondegenerate, 
follows in the same way.
A key point is that the proof of Theorem~\ref{thm:discretenondegen}
produces a finite witness set $S$ which is independent of $P$,
so long as $P$ has bounded degree.

\begin{remark}
It remains an open question whether the multilinear oscillatory
integral inequalities \eqref{decay} or \eqref{slowerdecay}
hold for all nondegenerate polynomial phases $P$, without additional hypotheses.
In the rationally commensurate case, 
Theorem~\ref{thm:periodicsublevelsets}
does rule out certain {\em strong} counterexamples to \eqref{slowerdecay}. 
Such a strong counterexample, for some sequence of values of $\lambda$
tending to $+\infty$, has each function $f_j(y)= f_{j,\lambda}(y)$ of the
form $f_{j,\lambda}(y) = e^{-i\phi_{j,\lambda}(y)}$ for some measurable real-valued 
phase $\phi_{j,\lambda}$,
with the phases satisfying $\distance(\lambda P(y)-\sum_j \phi_{j,\lambda}(\ell_j(y)),2\pi\integers)
< \delta(\lambda)$ for all $y$ outside a set $E_\lambda$,
where $\delta(\lambda)\to 0$ as $\lambda\to+\infty$ through the given sequence, 
and the measure $d\mu(y)=\eta(y)\,dy$ satisfies $\mu(E_\lambda)\to 0$
as $\lambda\to+\infty$.
In such a situation, the $\scripti_\lambda(f_1,\cdots,f_n)$ would not tend
to zero.
\end{remark}

\section{Discussion} \label{section:bilinearcase}

\noindent{\bf Bilinear case.}\ 
The bilinear case has been intensively studied. 
Recall first the nonsingular situation, in which
the mapping $y\mapsto (\ell_1(y),\ell_2(y))$
of $\reals^d$ to $\reals^{d_1}\times\reals^{d_2}$ is a bijection.
In this case, 
$\scripti_\lambda(f_1,f_2)$ can be written as
\[\int e^{i\lambda P(x,y)}f(x)g(y)\eta(x,y)\,dx\,dy,\]
where $P$ is a real-valued polynomial.
A necessary and sufficient for \eqref{decay}
is that there exist nonzero multi-indices
$\alpha,\beta$ for which
$\partial^{\alpha+\beta} P/\partial x^\alpha\partial y^\beta$
does not vanish identically; equivalently, $P$ is not
a sum of one function of $x$ plus another function of $y$.

Next consider the singular bilinear situation. $\scripti_\lambda(f,g)$
can always be expressed in the form
$\int e^{i\lambda P(x,y,z)}f(x,z)g(y,z)\eta(x,y,z)\,dx\,dy\,dz$,
where $x,y$ range over Euclidean spaces of arbitrary dimensions,
and $z$ over a space of positive dimension.
Such an expression satisfies \eqref{decay}
if and only if there exist $\alpha,\beta\ne 0$ such that
$\partial^{\alpha+\beta} P/\partial x^\alpha \partial y^\beta$
does not vanish identically as a function of all three variables.
For on one hand, if all such mixed partial derivatives
do vanish identically, then $P(x,y,z)$ can be decomposed in the form
$p(x,z)+q(y,z)$. The resulting factors $e^{i\lambda p}$ and $e^{i\lambda q}$
can be incorporated into $f,g$ respectively, and there is consequently no 
valid inequality \eqref{decay}.
On the other hand, if some such mixed partial derivative 
does not vanish identically, then
integration with respect to $x,y$ for fixed $z$ sets up a
nonsingular problem for a bilinear form. The 
result of the preceding paragraph gives a bound 
\[C\min\big(1, |Q(z)|^{-1}|\lambda|^{-\delta}\big)
\norm{f(\cdot,z)}_\infty\norm{g(\cdot,z)}_\infty\]
for some exponent $\delta>0$ and some polynomial $Q$ which does not
vanish identically. \eqref{decay} easily follows by integration with
respect to $z$.

The conclusion is that any formally singular bilinear situation
can be reduced to nonsingular ones by freezing some of the coordinates,
exploiting oscillation, then integrating with respect to the frozen coordinates.

\medskip
\noindent{\bf Higher order case.}\ 
In contrast,
the singular multilinear forms of higher order 
studied in this paper are not in general
reducible to nonsingular ones in this way. 
As an example,
define $P:\reals^3\to\reals$ to be
$P(x_1,x_2,x_3)=x_3^2$. 
Fix a large positive integer $N$, to be specified below.
For $j\in\{1,2,3,\cdots,N\}$
choose nonzero unit vectors $v_j = (v_j^1,v_j^2,v_j^3)\in\reals^3$,
none of which is a scalar multiple of another,
all satisfying 
\begin{equation} \label{wave}
(v_j^3)^2 = (v_j^1)^2+(v_j^2)^2.
\end{equation}
Define 
$\ell_j(x) = x\cdot v_j = x_1v^1_j+x_2v^2_j+x_3v^3_j$,
and consider 
\[\scripti_\lambda(f_1,f_2,\cdots,f_N)
= \int_{\reals^3}e^{i\lambda P(x)}\prod_{j=1}^N f_j(\ell_j(x))\,\eta(x)\,dx\]
where $\eta\in C^\infty_0$ is a cutoff function which does not vanish
identically.
This multilinear operator is singular, since the integral is taken
over $\reals^3$ but the sum of the dimensions of the target spaces
of the mappings $\ell_j$ is $N$.
The differential operator
$L = \frac{\partial^2}{\partial x_3^2}
-\frac{\partial^2}{\partial x_1^2}-\frac{\partial^2}{\partial x_2^2}$
annihilates $f_j\circ\ell_j$ for all $j$, by virtue of the equations
\eqref{wave}, but does not annihilate $P$.
Therefore $P$ is nondegenerate relative to $\{\ell_j: 1\le j\le N\}$.

Consider the 
restriction of $P$ and the $\ell_j$ to any two-dimensional affine
subspace $V$. 
Let $w_j$ be the projection
of $v_j$ onto the unique parallel translate of $V$ which contains $0$. 
For $x\in V$, $x\cdot v_j = x\cdot w_j$ plus a constant independent of $x$.
$V$ may identified with $\reals^2$, and
the integral over $V$ is then expressed as
$\int_{\reals^2}e^{iQ(y)}\prod_{j=1}^N (\tilde f_j(y\cdot w_j))
\,\tilde\eta(y)\,dy$,
where the phase $Q$ is a {\em quadratic} polynomial.

We claim that $\{v_j\}$ can be chosen
so that for every affine two-dimensional subspace $V$
of $\reals^3$, $P|_V$ is degenerate, relative to
$\{\ell_j|_V\}$.
Thus there is no inequality of the form \eqref{decay}, 
nor any sublevel set bound of the form 
\eqref{eq:mainbound},\eqref{eq:thetatozero},
relative to $V$.
Therefore the method of reduction to lower dimension
by ``slicing'' is not applicable.

To establish the claim, 
observe first that 
any quadratic polynomial $Q:\reals^2\to\reals$ is 
necessarily degenerate,
relative to any family of three or more mappings
of the form $L_j(y)=y\cdot w_j$ which satisfy the requirement
that none of the vectors $w_j$ is a scalar multiple of any of the others.
This is seen by permuting the indices $j$ and changing coordinates 
so that $w_1=(1,0)$, $w_2=(0,1)$, and $w_3=(a,b)$ with both $a,b$ nonvanishing.
Then any quadratic polynomial in $(y_1,y_2)$ can be
expressed as a linear combination of $\{y_1,y_2,y_1^2, y_2^2, (ay_1+by_2)^2\}$.

It remains only to show that $N$ and $\{v_j: 1\le j\le N\}$ can
be chosen both to satisfy \eqref{wave} and
so that for every two-dimensional subspace $V\subset\reals^3$,
some subcollection of three of the associated vectors $\{w_j: 1\le j\le N\}$ 
has no element equal to a scalar multiple of any other element.
Define 
\[
v_j =(v^1_j,v^2_j,v^3_j) = 2^{-1/2}(\cos(2\pi/N),\sin(2\pi/N),1).
\]
If $N$ is sufficiently large then the required property clearly holds,
for otherwise one can obtain a contradiction by letting $N\to\infty$
and exploiting the compactness of the Grassmann manifold of all subspaces
$V$.

\medskip
\noindent {\bf Slack.}\ 
There are at least two places in the analysis at which available information
has been only partially exploited.  Firstly, only translates and dilates of a single
finite point configuration were used to establish the sublevel set bounds,
while our algebraic discussion showed that translates and dilates
of a rather large family of configurations are actually excluded from
sublevel sets.
Secondly, the identity 
$\sum_s c_s (P-\sum_j f_j\circ\ell_j)(x+rs) =h(x,r)$
was used merely to obtain an inequality
$\sum_s |(P-\sum_j f_j\circ\ell_j)(x+rs)|\ge c |h(x,r)|$.

\medskip
\noindent{\bf Connection with Gowers uniformity norms.}\ 
The first step in the algebraic proof of the existence of finite witness sets
is to apply a finite difference operator to $P-\sum_j (f_j\circ\ell_j)$
which annihilates $e^{i\lambda P}$ and every term $f_j\circ\ell_j$ but a single one.
This operation has an analytic counterpart for multilinear oscillatory integrals. 
Write $\scripti_\lambda(f_1,\cdots,f_m) = \langle T_\lambda(f_1,\cdots,f_{m-1}),\,f_m\rangle$
for  certain multilinear operators $T_\lambda$. 
Then $|\scripti_\lambda(f_1,\cdots)|^2\lesssim \norm{f_m}_{\lt}^2
\int |T_\lambda(f_1,\cdots,f_{m-1})|^2$, and it suffices to
obtain an upper bound for the integral. This 
leads to the elimination of $f_m$,
the replacement of $P$ by a polynomial of lower degree, 
and the replacement of each remaining $f_j$ by $f_j(y+\ell_j(v))\overline{f_j(y)}$,
with an additional integration 
with respect to $v\in \text{nullspace}\,(\ell_m)$. 
One can iterate this operation until only $f_1$ remains, and then if necessary,
iterate finitely many additional times until no oscillatory factor $e^{i\lambda P}$ remains.

Suppose for simplicity that all the target spaces $\reals^{d_j}$ are one-dimensional.
One then obtains a bound of the type \eqref{slowerdecay} unless  
the Gowers uniformity norm $\norm{f_1}_{U^k(\reals^{d_1})}$
is bounded below by $\eta(\lambda)\norm{f_1}_{L^\infty}$
for a certain $k$, where $\eta(\lambda)\to 0$ very slowly as $|\lambda|\to\infty$.
The index $k$ which arises depends both on the degree of $P$,
and on the number of functions $f_j$.
This argument applies for each index $j$, so by multilinearity,
the estimation of $\scripti_\lambda$ reduces to the case
in which none of the functions $f_j/\|f_j\|_{L^\infty}$ has very small uniformity norm.

Thus an appropriate description of functions whose
uniformity norms are not small should lead to a proof of \eqref{slowerdecay}.
Certain descriptions are now known \cite{GTZ1},\cite{GTZ2}.
Their thrust is that if $\norm{f}_{U^k}$
is not small relative to $\norm{f}_{L^\infty}$, then
$f$ can be decomposed into a controlled sum of functions which
resemble $e^{iQ}$ for polynomials $Q$ of bounded degree,
plus a remainder with small $U^k$ norm.  
The results available when this paper was written were apparently quantitatively too weak to 
yield power decay bounds, but the more recent inverse theorems of Green, Tao, and Ziegler 
\cite{GTZ1},\cite{GTZ2} may be  more fruitful.
See \cite{taovu} for an introduction to these matters.

\begin{thebibliography}{15}

\bibitem{bcct1}
J.~Bennett, A.~Carbery, M.~Christ and T.~Tao,
{\em The Brascamp-Lieb inequalities: finiteness, structure and extremals},  
Geom. Funct. Anal.  17  (2008),  no. 5, 1343--1415.

\bibitem{bcct2}
\bysame,
{\em Finite bounds in H\"older-Brascamp-Lieb multilinear inequalities},
to appear, Math.\ Research Letters.
math.CA/0505691

\bibitem{ccw}
A.~Carbery, M.~Christ, and J.~Wright,
{\em Multidimensional van der Corput and sublevel set estimates},
J. Amer. Math. Soc. 12 (1999), no. 4, 981--1015.

\bibitem{trilinear}
M.~Christ,
{\em On the simplest trilinear operators},
Math. Research Letters 8 (2001), 43-56.

\bibitem{christreduction}
\bysame,
{\em Multilinear oscillatory integrals via reduction of dimension},
preprint.

\bibitem{christdosilva}
M.~Christ and D.~Oliveira do Silva,
{\em On trilinear oscillatory integrals},
preprint.

\bibitem{multiosc}
M.~Christ, X.~Li, T.~Tao, and C.~Thiele,
{\em On multilinear oscillatory integrals, nonsingular and singular},
Duke Math. J.  130 (2005), 321--351.



\bibitem{furstkatz}
H.~Furstenberg and Y.~Katznelson,
{\em An ergodic Szemer\'edi theorem for commuting transformations},
J. Analyse Math. 34 (1978), 275--291.

\bibitem{greenblattsmooth}
M.~Greenblatt,
{\em Simply nondegenerate multilinear oscillatory integral operators with smooth phase}, 
Math. Res. Lett. 15 (2008), no. 4, 653--660. 

\bibitem{GTZ1}
B.~Green, T.~Tao, and T.~Ziegler,
{\em An inverse theorem for the Gowers $U^{s+1}[N]$-norm (announcement)},
arXiv:1006.0205, math.NT (math.DS).

\bibitem{GTZ2}
\bysame,
{\em An inverse theorem for the Gowers $U^{s+1}[N]$-norm},
arXiv:1009.3998, math.CO (math.DS).




\bibitem{steinfatbook}
E.~M.~Stein,
{\em Harmonic analysis: real-variable methods, orthogonality, and oscillatory integrals},
With the assistance of Timothy S. Murphy. Princeton Mathematical Series, 43. Monographs in Harmonic Analysis, III. Princeton University Press, Princeton, NJ, 1993.

\bibitem{taovu}
T.~Tao and V.~Vu,
{\em Additive Combinatorics}, 
Cambridge Studies in Advanced Mathematics, 105. Cambridge University Press, Cambridge, 2010.

\end{thebibliography}
\end{document}